\documentclass[letterpaper]{article}
\usepackage{amsmath, amsthm, amssymb, amsfonts,amsrefs} 
\usepackage{mathtools}
\usepackage{enumerate}
\usepackage[font=small]{caption}
\usepackage{subcaption}
\usepackage{float}
\usepackage{tikz}
\usetikzlibrary{calc}
\usepackage{picture}
\usepackage{rotating}
\usepackage{comment}
\usepackage{graphicx}
\theoremstyle{plain}\newtheorem{theorem}{Theorem}
\theoremstyle{plain}\newtheorem{lemma}{Lemma}
\theoremstyle{plain}
\theoremstyle{plain}
\theoremstyle{plain}\newtheorem{claim}{Claim}
\theoremstyle{definition}\newtheorem{definition}{Definition}
\theoremstyle{definition}
\theoremstyle{definition}

\numberwithin{equation}{section}
\numberwithin{theorem}{section}
\numberwithin{definition}{section}
\numberwithin{cor}{section}
\numberwithin{prop}{section}
\numberwithin{example}{section}

\renewcommand{\epsilon}{\varepsilon}
\renewcommand{\phi}{\varphi}

\newcommand{\major}{1cm}
\newcommand{\minor}{0.5cm}
\newcommand{\ellipse}[3]{\draw (#1) ellipse (#2 and #3);}
	\newcommand{\typeH}[3]{\ellipse{#1}{#2}{#3}
	\coordinate (bleft) at ($(#1)+(-.6*\major, .80*\minor)$);
	\coordinate (bright) at ($(#1)+(.6*\major, .80*\minor)$);
	\coordinate (tleft) at ($(bleft)+(.1cm,.4cm)$);
	\coordinate (tright) at ($(bright)+(-.1cm,.4cm)$);
	\draw [thick, black] (bleft) -- (tleft);
	\draw [thick, black] (bright) -- (tright);
	\draw [thick, black] (tleft) -- (tright);
	\node[draw,circle,inner sep = 1pt,fill] at (tleft) {};
	\node[draw,circle,inner sep = 1pt,fill] at (tright) {};
	}
	\newcommand{\typeAnoEdge}[3]{\ellipse{#1}{#2}{#3}
	\coordinate (bleft) at ($(#1)+(-.6*\major, .80*\minor)$);
	\coordinate (bmiddle1) at ($(#1)+(-.2*\major, .98*\minor)$);
	\coordinate (bmiddle2) at ($(#1)+(.2*\major, .98*\minor)$);
	\coordinate (bright) at ($(#1)+(.6*\major, .80*\minor)$);
	\coordinate (tleft) at ($(bleft)+(.1cm,.4cm)$);
	\coordinate (tright) at ($(bright)+(-.1cm,.4cm)$);
	\draw [thick, black] (bleft) -- (tleft);
	\draw [thick, black] (bright) -- (tright);
	\draw [thick, black] (tleft) -- (bmiddle1);
	\draw [thick, black] (tright) -- (bmiddle2);
	\node[draw,circle,inner sep = 1pt,fill] at (tleft) {};
	\node[draw,circle,inner sep = 1pt,fill] at (tright) {};
	}
	\newcommand{\typeAEdge}[3]{\ellipse{#1}{#2}{#3}
	\typeAnoEdge{#1}{#2}{#3}
	\coordinate (v1) at ($(tright) + (0,.4)$);
	\coordinate (v2) at ($(v1) + (.5,0)$);
	\draw [thick, black] (v1) -- (v2);
	\node[draw,circle,inner sep = 1pt,fill] at (v1) {};
	\node[draw,circle,inner sep = 1pt,fill] at (v2) {};	
	}
	\newcommand{\typeA}[3]{\ellipse{#1}{#2}{#3}
	\typeAnoEdge{#1}{#2}{#3}
	\coordinate (cdots) at ($(#1)+(\major+2,0)$);
	\node[right, scale = 1.5] at (cdots) {$\cdots$} ;
	\typeAEdge{$(cdots)+(\major+1cm,0)$}{#2}{#3}
	}	
	\newcommand{\typeB}[3]{\ellipse{#1}{#2}{#3}
	\coordinate (bleft) at ($(#1)+(-.78*\major, .64*\minor)$);
	\coordinate (bmiddle1) at ($(#1)+(.24*\major, .95*\minor)$);
	\coordinate (bmiddle2) at ($(#1)+(.4*\major, .92*\minor)$);	
	\coordinate (bright) at ($(#1)+(.82*\major, .60*\minor)$);
	\coordinate (tleft) at ($(bleft)+(.1, .55cm)$);
	\coordinate (tmiddle) at ($(bmiddle1)+(-.1, .41cm)$);
	\coordinate (tright) at ($(bright)+(-.1cm, .6cm)$);
	\draw [thick, black] (bleft) -- (tleft);
	\draw [thick, black] (bmiddle1) -- (tmiddle);
	\draw [thick, black] (bmiddle2) -- (tright);
	\draw [thick, black] (bright) -- (tright);
	\draw [thick, black] (tleft) -- (tmiddle);
	\node[draw,circle,inner sep = 1pt,fill] at (tleft) {};
	\node[draw,circle,inner sep = 1pt,fill] at (tmiddle) {};
	\node[draw,circle,inner sep = 1pt,fill] at (tright) {};
	}
	\newcommand{\typeX}[3]{\ellipse{#1}{#2}{#3}
	\coordinate (bleft) at ($(#1)+(-.58*\major, .82*\minor)$);
	\coordinate (bright) at ($(#1)+(.58*\major, .82*\minor)$);
	\coordinate (tleft) at ($(bleft)+(0, .4cm)$);
	\coordinate (tright) at ($(bright)+(0, .4cm)$);
	\draw [thick, black] (bleft) -- (tleft);
	\draw [thick, black] (bright) -- (tright);
	\node[draw,circle,inner sep = 1pt,fill] at (tleft) {};
	\node[draw,circle,inner sep = 1pt,fill] at (tright) {};
	}
	\newcommand{\typeY}[3]{\ellipse{#1}{#2}{#3}
	\coordinate (bleft) at ($(#1)+(-.5*\major, .84*\minor)$);
	\coordinate (bmiddle) at ($(#1)+(.15*\major, .99*\minor)$);
	\coordinate (bright) at ($(#1)+(.62*\major, .79*\minor)$);
	\coordinate (tleft) at ($(bleft)+(0, .5cm)$);
	\coordinate (tright) at ($(bright)+(-.15cm, .5cm)$);
	\draw [thick, black] (bleft) -- (tleft);
	\draw [thick, black] (bmiddle) -- (tright);
	\draw [thick, black] (bright) -- (tright);
	\node[draw,circle,inner sep = 1pt,fill] at (tleft) {};
	\node[draw,circle,inner sep = 1pt,fill] at (tright) {};
	}

\title{On the Hamiltonicity of the $k$-regular graph game}
\author{Jeremy Meza and Samuel Simon}
\date{\today}

\begin{document}
\maketitle

\begin{abstract}
We consider a game played on an initially empty graph where two players alternate drawing an edge between vertices subject to the condition that no degree can exceed $k$. We show that for $k=3$, either player can avoid a Hamilton cycle, and for $k\geq4$, either player can force the resulting graph to be Hamiltonian.

\end{abstract}

\section{Introduction}

The Hamiltonicity of $k$-regular graphs has been studied extensively. Bollob\'{a}s \cite{BB} and Fenner and Frieze \cite{FF} showed that for sufficiently large $k$, almost all random $k$-regular graphs are Hamiltonian. Robinson and Wormald improved this by showing that almost all cubic graphs are Hamiltonian \cite{RWcubic}, and later that almost all $k$-regular graphs for $k \geq 3$ are Hamiltonian \cite{RW}. Note in contrast that almost all 2-regular graphs are not Hamiltonian.

There have also been various results regarding Hamiltonicity in Maker-Breaker games. For random graphs, Ben-Shimon, Ferber, Hefetz and Krivelevich \cite{BFHK} showed that with high probability, the Maker wins the Hamiltonian game exactly when the random graph process reaches minimum degree 4, whereas the breaker wins if the minimum degree is at most 3 (see also \cite{HKS}).  Lastly, considering the biased $(1: b)$ Hamiltonicity Maker-Breaker game on $K_n$, Krivelevich \cite{K} gave an upper bound of $ \left(1- \frac{30}{(\log n)^{1/4}} \right) \frac{n}{\log n} $ on the critical bias of the game, which is the maximum possible value of $b$ for which Maker still wins.

We show another type of Hamiltonian threshold. Consider the following 2-person game played on $n$ vertices:
one at a time, each player draws one edge between any two vertices, under the condition that no vertex can have more than degree $k$. This continues until no player can draw an edge without making a vertex of degree greater than $k$. This game was introduced by Frieze and Pegden \cite{FP}, in which they showed that in the 3-regular game, any player can ensure planarity of the resulting graph, whereas in the 4-regular game, any player can ensure the resulting graph has an arbitrarily large clique minor. In this paper we consider when players can force Hamiltonicity. In particular, we prove the following theorems.

\begin{theorem}
For $k \geq 4$, and regardless of who has the first move, a player in the $k$-regular graph game has a strategy to ensure the resulting graph is Hamiltonian.
\label{Hamilton cycle}
\end{theorem}
\begin{theorem}
Regardless of who has the first move, a  player in the 3-regular graph game has a strategy to ensure the resulting graph is not Hamiltonian.
\label{no Hamilton cycle}
\end{theorem}

\section{Forcing Hamiltonicity} \label{forcing hamiltonicity}

Call the player with the goal of making the resulting graph Hamiltonian the \emph{Hamiltonian player}; call the opponent the \emph{non-Hamiltonian player}.

\begin{proof}[Proof of Theorem \ref{Hamilton cycle}]

We give a strategy so that as long as isolated vertices remain, at the end of the Hamiltonian player's turn, there is a Hamilton path $P$ on the induced subgraph of non-isolated vertices with ends $x_1$ and $x_2$ of degree 1 and degree at most 2, respectively. Note that this implies that the Hamiltonian player can force a Hamilton cycle by drawing the edge $x_1 \sim x_2$ when no isolated vertices remain. We proceed by induction on the number of moves the Hamiltonian player has taken.

The base case depends on which of the players moves first. If the Hamiltonian player moves first, then the single edge drawn is a Hamilton path as desired. If the non-Hamiltonian player moves first, the Hamiltonian player will draw an edge from one of the endpoints of the already existing edge to an isolated vertex, thereby creating a Hamilton path as desired.

Now suppose after the Hamiltonian player's $m^{th}$ move, there is a Hamilton path $P$ as described above. We use the following notation: $p_i \in P$ denotes a vertex in $P$ that is not an end, $x_1, x_2$ denote the ends of $P$ with degrees 1 and at most 2, respectively, and $v, w$ denote vertices not in $P$. Through exhaustion of cases, we give the corresponding move the Hamiltonian player should make in order to maintain the desired Hamilton path.

\begin{center}
\begin{tabular}{l | l}
non-Hamiltonian player's move & Hamiltonian player's move \\ \hline
(a) Draw an edge from $v$ to $w$.				& (a) Draw an edge from $x_2$ to $v$. \\
(b) Draw an edge from $p_i$ to $p_j$, $i \neq j$. 	& (b) Draw an edge from $x_2$ to $v$. \\
(c) Draw an edge from $p_i$ to $x_j$, $j=1,2$. 		& (c) Draw an edge from $x_j$ to $v$. \\
(d) Draw an edge from $p_i$ to $v$. 				& (d) Draw an edge from $x_2$ to $v$. \\
(e) Draw an edge from $x_j$ to $v$, $j=1,2$.		& (e) Draw an edge from $v$ to $w$. \\
(f) Draw an edge from $x_1$ to $x_2$.			& (f) Draw an edge from $x_2$ to $v$.
\end{tabular}
\end{center}

Figure \ref{HP strategy} depicts the 6 possible moves described above that the non-Hamiltonian player can make and the corresponding moves the Hamiltonian player will make in order to maintain the desired Hamilton path. \end{proof}

\begin{figure}[h]
\begin{subfigure}[h]{0.3\textwidth}
\begin{tikzpicture}

    \coordinate (e1) at (0,0);
    \coordinate (e2) at (3,0);
    \coordinate (halfway) at (2,0);
    \coordinate (v) at (2,1);
    \coordinate (w) at (3,1);

    \node[draw,circle,inner sep=1pt,fill] at (e1) {};
    \node at (e1) [below left] {$x_1' = x_1$};
    \node[draw,circle,inner sep=2pt,fill] at (e2) {};
    \node[draw,circle,inner sep=1pt,fill] at (v) {};
    \node[draw,circle,inner sep=1pt,fill] at (w) {};
     
    \draw [very thick, black] (e1) -- (e2) node [below right] {$x_2$};
    \draw [dashed] (w) -- (v) node [above left] {$x_2' = w$};
    \draw [thick, dotted] (e2) -- (w) node [above right] {$v$};

\end{tikzpicture}
\subcaption{nHP connects $v \sim w$. }
\end{subfigure}
\hspace{2.5cm}
\begin{subfigure}[h]{0.3\textwidth}
\begin{tikzpicture}

    \coordinate (e1) at (0,0);
    \coordinate (e2) at (3,0);
    \coordinate (halfway) at (2,0);
    \coordinate (pi) at (1,0);
    \coordinate (pj) at (2.5,0);
    \coordinate (v) at (3,1);

    \node[draw,circle,inner sep=1pt,fill] at (e1) {};
    \node at (e1) [below left] {$x_1' = x_1$};
    \node[draw,circle,inner sep=2pt,fill] at (e2) {};  
    \node[draw,circle,inner sep=1pt,fill] at (pi) {};
    \node at (pi) [below] {$p_i$};
    \node[draw,circle,inner sep=1pt,fill] at (pj) {};
    \node at (pj) [below] {$p_j$};
    \node[draw,circle,inner sep=1pt,fill] at (v) {};
         
    \draw [very thick, black] (e1) -- (e2) node [below right] {$x_2$};
    \draw [dashed] (pj) arc [radius=1.1, start angle=45, end angle= 135];
    \draw [thick, dotted] (e2) -- (v) node [above right] {$v=x_2'$};

\end{tikzpicture}
\subcaption{nHP connects $p_i \sim p_j$. }
\end{subfigure}
\\
\begin{subfigure}[h]{0.3\textwidth}
\begin{tikzpicture}
   
    \coordinate (e1) at (0,0);
    \coordinate (e2) at (3,0);
    \coordinate (halfway) at (2,0);
    \coordinate (pi) at (1,0);
    \coordinate (v) at (3,1);

    \node[draw,circle,inner sep=1pt,fill] at (e1) {};
    \node at (e1) [below left] {$x_i' = x_i$};
    \node[draw,circle,inner sep=1pt,fill] at (e2) {};  
    \node[draw,circle,inner sep=1pt,fill] at (pi) {};
    \node at (pi) [below] {$p_i$};
    \node[draw,circle,inner sep=1pt,fill] at (v) {};
         
    \draw [very thick, black] (e1) -- (e2) node [below right] {$x_j$};
    \draw [dashed] (e2) arc [radius=1.43, start angle=45, end angle= 135];
    \draw [thick, dotted] (e2) -- (v) node [above right] {$v=x_j'$};

\end{tikzpicture}
\subcaption{nHP connects $p_i \sim e_j$, $j \in \{1,2\}$. }
\end{subfigure}
\hspace{2.5cm}
\begin{subfigure}[h]{0.3\textwidth}
\begin{tikzpicture}

    \coordinate (e1) at (0,0);
    \coordinate (e2) at (3,0);
    \coordinate (halfway) at (2,0);
    \coordinate (pi) at (1,0);
    \coordinate (v) at (1,1);

    \node[draw,circle,inner sep=1pt,fill] at (e1) {};
    \node at (e1) [below left] {$x_1'=x_1$};
    \node[draw,circle,inner sep=2pt,fill] at (e2) {};  
    \node[draw,circle,inner sep=1pt,fill] at (pi) {};
    \node at (pi) [below] {$p_i$};
    \node[draw,circle,inner sep=1pt,fill] at (v) {};
         
    \draw [very thick, black] (e1) -- (e2) node [below right] {$x_2$};
    \draw [dashed] (pi)--(v) node [above right] {};
    \draw [thick, dotted] (e2) -- (v) node [above right] {$v = x_2'$};

\end{tikzpicture}
\subcaption{nHP connects $p_i \sim v$. }
\end{subfigure}
\\
\begin{subfigure}[h]{0.3\textwidth}
\begin{tikzpicture}

    \coordinate (e1) at (0,0);
    \coordinate (e2) at (3,0);
    \coordinate (halfway) at (2,0);
    \coordinate (v) at (3,1);
    \coordinate (w) at (2,1);

    \node[draw,circle,inner sep=1pt,fill] at (e1) {};
    \node at (e1) [below left] {$x_i'=x_i$};
    \node[draw,circle,inner sep=1pt,fill] at (e2) {};  
    \node[draw,circle,inner sep=1pt,fill] at (v) {};
    \node[draw,circle,inner sep=1pt,fill] at (w) {};
             
    \draw [very thick, black] (e1) -- (e2) node [below right] {$x_j$};
    \draw [dashed] (e2)--(v) node [above right] {$v$};
    \draw [thick, dotted] (v) -- (w) node [above left] {$w=x_j'$};

\end{tikzpicture}
\subcaption{nHP connects $x_j \sim v$ for $j \in \{1,2\}$. }
\end{subfigure}
\hspace{2.5cm}
\begin{subfigure}[h]{0.3\textwidth}
\begin{tikzpicture}

    \coordinate (e1) at (0,0);
    \coordinate (e2) at (3,0);
    \coordinate (halfway) at (2,0);
    \coordinate (v) at (3,1);

    \node[draw,circle,inner sep=1pt,fill] at (e1) {};
    \node at (e1) [below left] {$x_2'=x_1$};
    \node[draw,circle,inner sep=2pt,fill] at (e2) {};  
    \node[draw,circle,inner sep=1pt,fill] at (v) {};
             
    \draw [very thick, black] (e1) -- (e2) node [below right] {$x_2$};
    \draw [dashed] (e2) arc [radius=2.15, start angle=45, end angle= 135];
    \draw [thick, dotted] (e2)--(v) node [above right] {$v=x_1'$};

\end{tikzpicture}
\subcaption{nHP connects $x_1 \sim x_2$. }
\end{subfigure}

\caption{The 6 possible moves the non-Hamiltonian player (nHP) can make (shown in dashed) and the corresponding moves the Hamiltonian player (HP) should make (shown in dotted) in order to preserve the desired Hamilton path. The original Hamilton path $P$ is depicted as the thick line. The smaller vertex denotes $x_1$ (of degree 1) and the bigger vertex denotes $x_2$ (of degree at most 2). Vertices $x_1', x_2'$ denote the ends of $P'$, the new Hamilton path after the Hamiltonian player's move.}
\label{HP strategy}
\end{figure}
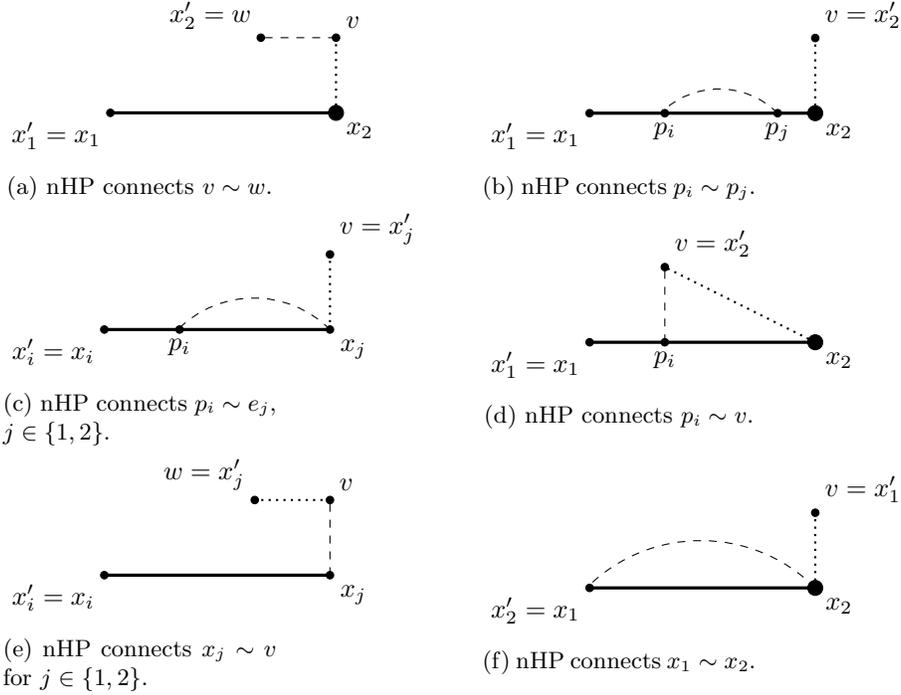

\section{Avoiding Hamiltonicity} \label{avoiding hamiltonicity}

We now prove Theorem \ref{no Hamilton cycle}, which gives a sharp Hamiltonian threshold. In order to do so, we prove a more general result.
\begin{theorem}
Given any initial graph $G_0$, there is some finite number of vertices $N$ so that if Players 1 and 2 play the $3$-regular graph game on $G_0$ and the empty graph on $N$ vertices, then Player 2 has a strategy that ensures the resulting graph is not 2-connected.
\label{general no Hamilton}
\end{theorem}
In the case when the non-Hamiltonian player makes the second move, then Theorem \ref{general no Hamilton} with $G_0 = \emptyset$ implies Theorem \ref{no Hamilton cycle}, since 2-connectedness is a necessary condition for Hamiltonicity. In the case when the non-Hamiltonian player makes the first move, then Theorem \ref{general no Hamilton} with $G_0$ being two vertices and an edge between them implies Theorem \ref{no Hamilton cycle}.

\begin{definition}
A vertex $x \in G$ is an \emph{eventual cut vertex} if in $G \setminus x$ there is a connected component $S$ such that $x$ has either 1 or 2 neighbors in $S$ and for all $v \in S$, $\deg(v) = 3$ in $G$.
\end{definition}

Note that if Players 1 and 2 are playing the 3-regular graph game, then at any stage in the game if there is an eventual cut vertex and at least two vertices not connected, then the resulting graph cannot contain a Hamilton cycle.

\begin{definition}
Let $G$ be a connected graph. We say that $G$ is \emph{type $\mathcal{H}$} if $G$ is 3-regular except for two vertices of degree 2 that are adjacent.
\end{definition}
Figure \ref{type H ex} depicts examples of graphs with an eventual cut vertex and an example of a type $\mathcal{H}$ graph. The ellipse represents a connected component with vertices of degree 3, and the vertices on top represent the vertices in the components that are not of degree 3.
\begin{figure}[h]
\centering
\begin{subfigure}[h]{0.5\textwidth}
\centering
\begin{tikzpicture}
	\ellipse{0,0}{\major}{\minor}
	\coordinate (bleft) at ($(0,0)+(-0.3*\major, 0.95*\minor)$);
	\coordinate (bright) at ($(0,0)+(0.3*\major, 0.95*\minor)$);
	\coordinate (top) at ($(0,0)+(0, \minor+.5cm)$);
	\draw [thick, black] (bleft) -- (top);
	\draw [thick, black] (bright) -- (top);
	\node[draw,circle,inner sep = 1pt,fill] at (top) {};
	
	\ellipse{3,0}{\major}{\minor}
	\coordinate (bottom) at ($(3,0)+(0, \minor)$);
	\coordinate (top2) at ($(3,0)+(0, \minor+.5cm)$);
	\draw [thick, black] (bottom) -- (top2);
	\node[draw,circle,inner sep = 1pt,fill] at (top2) {};
	
\end{tikzpicture}
\subcaption{Graphs with an eventual cut vertex}
\end{subfigure}
\hspace{1cm}
\begin{subfigure}[h]{0.3\textwidth}
\centering
\begin{tikzpicture}
	\typeH{0,0}{\major}{\minor}
\end{tikzpicture}
\subcaption{Type $\mathcal{H}$ graph}
\end{subfigure}
\caption{ \label{type H ex} }
\end{figure}

\begin{proof}[Proof of Theorem \ref{general no Hamilton}]

Fix a vertex $x \in G_0$.
Define $G_n$ to be the graph after the $n^{th}$ move, excluding isolated vertices. Define $C_n$ to be the connected component of $G_n$ containing $x$ and let $D_n = G_n \setminus C_n$.

For $n \geq 1$, after the $(2n-1)^{th}$ move, Player 2's strategy is as follows:

If $C_{2n-1}$ has less than four vertices, then draw an edge from a vertex in $C_{2n-1}$ to any isolated vertex. Otherwise, the strategy is as follows (Player 2 always draws edges within $C_{2n-1}$ unless specifically stated otherwise): \\[.5em]
\begin{tabular}{p{5.9cm} | p{5.38cm} }
State of $C_{2n-1} \quad (|C_{2n-1}| \geq 4)$ & Player 2's move \\ \hline
(a) There is more than one vertex of degree 1. &  Draw an edge between any two vertices of degree 1. \\
(b) There is exactly one vertex $w$ of degree 1 and exactly two vertices $u, v$ of degree 2 such that $u \not \sim v$ and WLOG $w \sim u$ and $w \not \sim v$. & Draw the edge $w \sim v$.  \\
(c) There is exactly one vertex of degree 1 and not in option (b). & Draw an edge between the vertex of degree 1 and a vertex of degree 2. \\
(d) There are no vertices of degree 1 and two vertices of degree 2 that are not adjacent. & Draw and edge between two vertices of degree 2 such that this does not leave two degree 2 vertices that are adjacent (cf. Lemma~\ref{square}). \\
(e) $C_{2n-1}$ is type $\mathcal{H}$ and $E(D_{2n-1}) \neq 0$. & Draw an edge from one of the vertices of degree 2 to a vertex in $D_{2n-1}$. \\
(f) $C_{2n-1}$ is type $\mathcal{H}$ and $E(D_{2n-1}) = 0$. & Strategy outlined in Claim \ref{tree of cases}.
\end{tabular}

\vspace{5pt}
First note that the state of $C_{2n-1}$ is partitioned by the number of degree 1 vertices, and when there are no degree 1 vertices, is further partitioned by the number of degree 2 vertices. Every case is covered except for when there are no vertices of degree 1 and less than two vertices of degree 2. However, in this case, then $C_{2n-1}$ is either 3-regular or has an eventual cut vertex, and so already witnesses the non-Hamiltonicity of the resulting graph.

Now note that there are cases in which Player 2's strategy is not possible. For example, if $C_{2n-1}$ is in option (c), but there are no degree 2 vertices to connect the degree 1 vertex to, then Player 2 cannot follow this strategy.
We first claim that if for some $n$, Player 2's strategy is not possible to perform on $C_{2n-1}$, then $C_{2n-1}$ witnesses the non-Hamiltonicity of the final graph $G$.
This is because if Player 2 cannot follow the strategy, then $C_{2n-1}$ must be in one of the following cases:

\noindent Case 1: $C_{2n-1}$ is 3-regular.

\noindent Case 2: $C_{2n-1}$ is 3-regular except for exactly one vertex of degree 1 or 2. 

\noindent Case 3: $C_{2n-1}$ is 3-regular except for exactly one vertex of degree 1 and exactly one vertex of degree 2 that are adjacent.

In all of these cases, $C_{2n-1}$ is either 3-regular of has an eventual cut vertex, and hence $G_n$ and $G$ cannot have a Hamilton cycle.

Before we show that Player 2's strategy is a winning strategy, we first introduce some concepts.
For any connected component $A$, define the \emph{degrees of freedom} in $A$, $F(A)$ as
\[ F(A) := \sum_{v \in A} 3 - \deg(v) \]
Define the \emph{effective degrees of freedom} $E(A) := F(A) - 2 $.
The degrees of freedom and effective degrees of freedom of a graph is the sum of the degrees of freedom and effective degrees of freedom of each component, respectively.

We now show the following two claims.
\begin{claim}
There exists $n \geq 1$ so that one of the following will be true:
\begin{enumerate}[(i)]
\itemsep-.2em
\item $C_{n}$ is type $\mathcal{H}$ and $E(D_n) = 0$.
\item $G_{n}$ has a component that is 3-regular.
\item $G_{n}$ has a component with an eventual cut vertex.
\end{enumerate}
\label{3 options}
\end{claim}
\begin{claim}
Suppose that for some $n$, $C_{n}$ is of type $\mathcal{H}$, and $E(D_n) = 0$.
Then there exists $m > n$ such that $G_m$ has a component that is either 3-regular or has an eventual cut vertex.
\label{tree of cases}
\end{claim}

Claim 1 gives $n$ such that either $G_n$ has a component that is 3 regular, has a component that has an eventual cut vertex, or $C_n$ is type $\mathcal{H}$ with $E(D_n) = 0$.
In the last case, claim 2 gives $m > n$ such that $G_m$  has a component that is either 3 regular or has an eventual cut vertex.
Subject to the proofs of claims 1 and 2, defining $N = m$ finishes the proof of Theorem \ref{general no Hamilton}
\end{proof}

\begin{proof}[Proof of Claim \ref{3 options}]

Suppose that $C_n$ is type $\mathcal{H}$ for some $n \geq 1$. We invoke the following lemmas, to be proved in Section \ref{case studies}.
\begin{lemma}
Let $C_{n}$ be type $\mathcal{H}$ without an eventual cut vertex and let $E(D_{n})=y$. Suppose there exists $m > n+1$ such that $C_m$ is type $\mathcal{H}$ without an eventual cut vertex.
Then, either 
\begin{enumerate}
\setlength{\itemsep}{-.2em}
\item $E(D_m) < y$ or
\item $E(D_m) = y$ and if there exists $m' > m+1$ with $C_{m'}$ of type $\mathcal{H}$ and no eventual cut vertex, then $E(D_{m'}) < y$.
\end{enumerate}
\label{E decreases}
\end{lemma}
\begin{lemma}
Suppose that Player 2 draws an edge within $C_{2n-1}$. Then,
\begin{equation} 
F(C_{2n+1}) + E(D_{2n+1}) \leq F(C_{2n-1}) + E(D_{2n-1}) 
\label{eq: weakly decreasing} 
\end{equation}
with equality attained if and only if Player 1 draws an edge between two isolated vertices.
\label{F+E non-increasing}
\end{lemma}
Let $S = \{n \mid C_n \text{ is type } H \}$. We can assume that $C_n$ does not have an eventual cut vertex for all $n \in S$, otherwise the claim is satisfied. Enumerate $S$ as $m = s_1, \ldots, s_L = M$.
Let $F(C_0) + E(D_0) = k$. By Lemma \ref{E decreases} and Lemma \ref{F+E non-increasing}, we can assume that $|S| \leq 2k-2$. Otherwise, applying Lemma \ref{E decreases} repeatedly, we get that
\[ E(D_M) \leq E(D_{s_{L-2}}) -1 \leq \cdots \leq E(D_{s_{L-(2k-2)}}) - (k-1) \leq E(D_m) - (k-1) \]
Lemma \ref{F+E non-increasing} implies that $F(C_m) + E(D_m) \leq k$.
This combined with the fact that $F(C_M) = F(C_m) = 2$ implies that
\begin{equation}
F(C_M) + E(D_M) \leq F(C_m) + E(D_m) - (k-1) \leq 1
\label{F+E less than 1}
\end{equation}
Note that if $F(C_M) \leq 1$, then $C_M$ is either 3-regular or has an eventual cut vertex. Similarly if $E(D_M) < 0$, then $D_M$ has a component that is either 3-regular or has an eventual cut vertex. Hence, \eqref{F+E less than 1} implies that some component in $G_M$ is 3-regular or has an eventual cut vertex.

Now, if for any $n \in S$, $E(D_n) = 0$, then (i) is satisfied. Assume that $E(D_n) \neq 0$ for all $n \in S$.
Since $M$ was maximal, then $C_n$ will not be type $\mathcal{H}$ for any $n > M$.

Fix $n$ so that $2n-1 > M$ and suppose $C_{2n-1}$ is not already 3-regular or has an eventual cut vertex. Since $C_{2n-1}$ is not type $\mathcal{H}$, then Player 2's strategy is always to draw an edge within $C_{2n-1}$.
Thus, by Lemma \ref{F+E non-increasing}, we have that \eqref{eq: weakly decreasing} holds for all $n$ with $2n-1 > M$ such that $C_{2n-1}$ is not 3-regular or has an eventual cut vertex.

Let $K = F(C_{M+1}) + E(D_{M+1})$.
Choose $N > M + 2 K$. Define the sequence $(a_n) = F(C_{2n-1}) + E(D_{2n-1})$ for $n$ with $M < 2n-1 < N$. 
If $C_{2n-1}$ is 3-regular or has an eventual cut vertex, then either (ii) or (iii) is satisfied. Assuming not, then note that $(a_n)$ is a non-increasing sequence of integers with initial value $\leq K$.
This implies that either there exists $n$ so that $a_n < 0$, or $a_n$ is constant at least $K$ times.

In the latter case, since equality in \eqref{eq: weakly decreasing} is attained iff Player 1 draws an edge between two isolated vertices, then $E(D_{2n-1})$ is strictly increasing at least $K$ times.
In order for $a_n$ to remain constant, then $F(C_{2n-1})$ must be strictly decreasing at least $K$ times. But this is impossible, since $F(C_{2n-1})$ is a non-negative number for all $n$ and $F(C_{2n-1}) \leq K$. Thus, we must be in the case $a_n < 0$ for some $n$.

If $a_n < 0$, then this implies that $E(D_{2n-1}) < 0$, since $F(C_{2n-1}) \geq 0$ for all $n$. This in turn implies that $D_{2n-1}$ has a component that is either 3-regular or has an eventual cut vertex, and so either (ii) or (iii) holds in our claim.

Thus, we see that if (i) is not satisfied, then (ii) or (iii) must be satisfied. \end{proof}

Before we prove Claim \ref{tree of cases}, thereby finishing the proof of Theorem \ref{general no Hamilton}, we introduce a few more concepts.
\begin{definition}
Let $G$ be a graph. We say that $G$ is \emph{type $\mathcal{A}$} if $G$ has a component that is a single edge, and every other component of $G$ is 3-regular except for two vertices of degree 2 that are not adjacent.
\end{definition}
\begin{definition}
Let $G$ be a connected graph. We say that $G$ is \emph{type $\mathcal{B}$} if $G$ is 3-regular except for three vertices of degree 2, exactly two of which are adjacent.
\end{definition}

\begin{figure}[h]
\centering
\begin{subfigure}[h]{0.3\textwidth}
\centering
\begin{tikzpicture}
	\typeA{0,0}{\major}{\minor}
\end{tikzpicture}
\subcaption{Type $\mathcal{A}$}
\end{subfigure}
\hspace{3cm}
\begin{subfigure}[h]{0.3\textwidth}
\centering
\begin{tikzpicture}
	\typeB{2,0}{\major}{\minor}
\end{tikzpicture}
\subcaption{Type $\mathcal{B}$}
\end{subfigure}
\caption{Depictions of Type $\mathcal{A}$ and Type $\mathcal{B}$ graphs.}
\label{types A and B ex}
\end{figure}
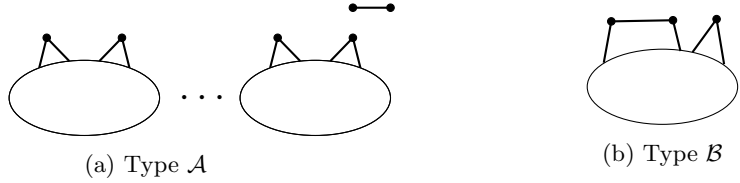 

\begin{proof}[Proof of Claim \ref{tree of cases}]
We case Player 2's strategy based on whose turn it is when $C_n$ is type $\mathcal{H}$ and $E(D_n) = 0$. First note that if $E(D_n) = 0$, then either $E(D_n^j) = 0$ for all components $D_n^j \subseteq D_n$, or there is a component $D_n^j$ with $E(D_n^j)<0$. In the latter case, this component has either 0 or 1 degrees of freedom, in particular it has an eventual cut vertex or is 3-regular, and so the claim is satisfied.

Thus, we consider the case where $E(D_n^j)=0$ for all components $D_n^j \subseteq D$. Thus, there are two degrees of freedom in each component, i.e. there are either two vertices of degree 2, or one vertex of degree one in $D_n^j$. We can assume each component has two vertices of degree 2, since one vertex of degree 1 is an eventual cut vertex, and so the claim is satisfied.

Suppose it is Player 2's turn. Let $v, w \in C_n$ be the vertices of degree 2, and let $p_j, q_j \in D_n^j$ be the vertices of degree 2.
Player 2's strategy is as follows:
\[ \begin{tabular}{l | l}
State of $G_n$ & Player 2's move \\ \hline
(a) If $D_n$ is not empty. & draw an edge from $v$ to $p_j$ for some $j$. \\
(b) If $D_n$ is empty. & Follow Figure \ref{tree of cases P2}.
\end{tabular}
\]
We show in each case that Player 2's strategy is a winning strategy:
\begin{enumerate}[(a)]
\item Observe that if an edge is drawn from $w$ to $q_j$, then this creates a 3-regular component. As Player 1 wants to avoid this, Player 1 must draw an edge off of either $w$ or $q_j$ to some vertex $a$, say without loss Player 1's move is $w \sim a$. If $a$ is a free vertex, then Player 2 can now draw an edge from $q_j$ to $a$, which makes $a$ an eventual cut vertex. Otherwise, $a$ is in some component $D_n^i$. However, we know that all components $D_n^i$ are 3-regular except for two vertices of degree 2, Thus there is a vertex $b$ of degree 2 in $D_n^i$. Drawing an edge between $q_j$ and $b$ makes a 3-regular component, which consists of the vertices from $C_n, D_n^i$, and $D_n^j$. Hence, the claim is satisfied.
\item Following Figure \ref{tree of cases P2}, we see that the end result is that there is $m > n$ such that $G_m$ is either type $\mathcal{A}$ or type $\mathcal{B}$. Lemmas \ref{type A} and \ref{type B}, which we prove in section \ref{case studies}, show that there is $m' > m$ such that $G_{m'}$ has a component that is 3-regular or has an eventual cut vertex, and hence the claim is satisfied.
\end{enumerate}
\begin{lemma}
Suppose that for some $n$, $G_n$ is type $\mathcal{A}$. Then, there exists $m > n$ such that $G_m$ has a component that is 3-regular or has an eventual cut vertex.
\label{type A}
\end{lemma}
\begin{lemma}
Suppose that for some $n$, $G_n$ is type $\mathcal{B}$. Then, there exists $m > n$ such that $G_m$ has a component that is 3-regular or has an eventual cut vertex.
\label{type B}
\end{lemma}

Now, if it is Player 1's turn, then Player 2's strategy is as follows: \\
\centering
\begin{tabular}{p{5.0cm} | p{5.7cm}}
Player 1's move & Player 2's move \\ \hline
(a) Draw an edge from $p_i$ to $p_j$. & Draw an edge from $q_i$ to $q_j$. \\
(b) Draw an edge from $v$ to $p_i$. & Draw an edge from $w$ to $q_i$ \\
(c) Draw an edge from $p_i$ to $q_i$. & Draw an edge between two isolated vertices. \\
(d) Draw an edge from $v$ to an isolated vertex $x$. & Draw an edge from $w$ to $x$. \\
(e) Draw an edge from $p_i$ to an isolated vertex $x$. & Draw an edge from $q_i$ to $x$. \\
(f) Draw an edge between two isolated vertices $x$ and $y$. & Draw an edge from $v$ to $p_i$ if there is a component $D_n^i$. Otherwise follow Figure \ref{tree of cases P1}. \\
\end{tabular} \\
Again, we show in each case that Player 2's strategy is a winning strategy:
\begin{enumerate}[(a)]
\setlength{\itemsep}{.07em}
\item This results in a 3-regular component that consists of $D_n^i$ and $D_n^j$.
\item This results in a 3-regular component that consists of $C_n$ and $D_n^i$.
\item Once Player 1 draws an edge from $p_i$ to $q_i$, then $D_n^i$ is already 3-regular, and so Player 2 can make any move.
\item This results in $x$ becoming an eventual cut vertex.
\item This results in $x$ becoming an eventual cut vertex.
\item Drawing an edge from $v$ to $p_i$ if there is a component $D_n^i$ will result in a type $\mathcal{A}$ graph. Following Figure \ref{tree of cases P1} will result in either a type $\mathcal{B}$ graph or will take us to Figure \ref{tree of cases P2}. Player 2 will win from any of these positions.
\end{enumerate}
\end{proof}


	\newcommand{\typeD}[3]{\ellipse{#1}{#2}{#3}
	\coordinate (bleft) at ($(#1)+(-.78*\major, .64*\minor)$);
	\coordinate (bmiddle) at ($(#1)+(.24*\major, .95*\minor)$);
	\coordinate (bright) at ($(#1)+(.65*\major, .78*\minor)$);
	\coordinate (tleft) at ($(bleft)+(.1, .55cm)$);
	\coordinate (tmiddle) at ($(bmiddle)+(-.1, .41cm)$);
	\coordinate (tright) at ($(bright)+(0, .5cm)$);
	\draw [thick, black] (bleft) -- (tleft);
	\draw [thick, black] (bmiddle) -- (tmiddle);
	\draw [thick, black] (bright) -- (tright);
	\draw [thick, black] (tleft) -- (tmiddle);
	\node[draw,circle,inner sep = 1pt,fill] at (tleft) {};
	\node[draw,circle,inner sep = 1pt,fill] at (tmiddle) {};
	\node[draw,circle,inner sep = 1pt,fill] at (tright) {};
	}
	\newcommand{\typeE}[3]{\ellipse{#1}{#2}{#3}
	\typeH{#1}{#2}{#3}
	\coordinate (tmiddle) at ($(tleft)!0.5!(tright)$);
	\coordinate (v1) at ($(tright) + (0,.5)$);
	\coordinate (v2) at ($(v1) + (.5,0)$);
	\draw [thick, black] (v1) -- (v2);
	\node[draw,circle,inner sep = 1pt,fill] at (tmiddle) {};
	\node[draw,circle,inner sep = 1pt,fill] at (v1) {};
	\node[draw,circle,inner sep = 1pt,fill] at (v2) {};	
	}


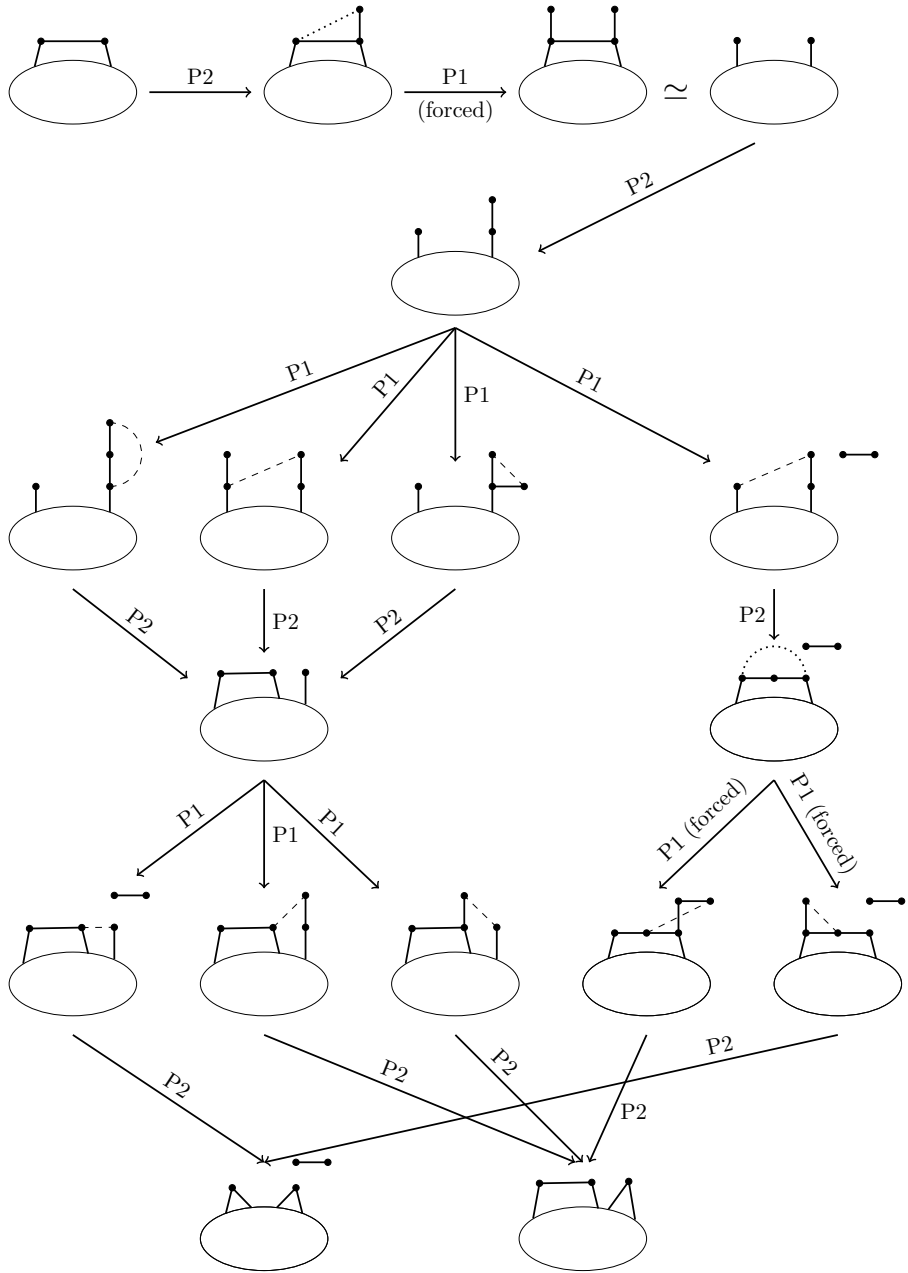
\begin{figure}[H]
{ \centering
\resizebox{.61\height}{!}{
\begin{tikzpicture}
	\typeH{0,3}{\major}{\minor}
	\draw [thick, black, ->] (\major+.2cm, 3) -- (\major+1.8cm, 3) node [above, midway] {P2};
	\newcommand{\drawgenII}[3]{
	\typeH{#1}{#2}{#3}
	\coordinate (v1) at  ($(tright)+(0,.5)$);
	\node[draw,circle,inner sep = 1pt,fill] at (v1) {};
	\draw [thick, black] (tright) -- (v1);
	}
	\drawgenII{4,3}{\major}{\minor}
	\draw [thick, black, dotted] (tleft) -- (v1);
	\draw [thick, black, ->] (4cm+\major+.2cm, 3) -- (4cm+\major+1.8cm, 3) node [above, midway] {P1} node [below, midway] {(forced)};
	\newcommand{\drawgenIII}[3]{
	\drawgenII{#1}{#2}{#3}
	\coordinate (v2) at  ($(tleft)+(0,.5)$);
	\node[draw,circle,inner sep = 1pt,fill] at (v2) {};
	\draw [thick, black] (tleft) -- (v2);
	\coordinate (iso) at ($(#1)+(\major+2,0)$);
	\node[right, scale = 1.5] at (iso) {$\simeq$} ;
	\typeX{$(#1)+(3,0)$}{#2}{#3}
	}
	\drawgenIII{8,3}{\major}{\minor}
	\draw [thick, black, ->] (11cm-.3cm, 3cm-\minor-.3cm) -- (6cm+\major+.3cm, .5cm) node [above, sloped, midway] {P2};
	\newcommand{\drawgenIV}[3]{
	\typeX{#1}{#2}{#3}
	\coordinate (v1) at ($(tright)+(0, .5)$);
	\node[draw,circle,inner sep = 1pt,fill] at (v1) {};
	\draw [thick, black] (tright) -- (v1);
	}
	\drawgenIV{6,0}{\major}{\minor}
	\draw [thick, black, ->] (6cm, -\minor-.2cm) -- (\major+.3cm, -2.5) node [above, sloped, midway] {P1};
	\draw [thick, black, ->] (6cm, -\minor-.2cm) -- (3cm+\major+.2cm, -2.8) node [above, sloped, midway] {P1};
	\draw [thick, black, ->] (6cm, -\minor-.2cm) -- (6cm, -2.8) node [right, midway] {P1};
	\draw [thick, black, ->] (6cm, -\minor-.2cm) -- (11cm-\major, -2.8) node [above, sloped, midway] {P1};
	\newcommand{\drawgenVa}[3]{
	\drawgenIV{#1}{#2}{#3}
	\coordinate (v2) at ($(v1)+(0, .5)$);
	\node[draw,circle,inner sep = 1pt,fill] at (v2) {};
	\draw [thick, black] (v1) -- (v2);
	}
	\drawgenVa{0,-4}{\major}{\minor}
	\draw [black, dashed] (tright) arc (270:450: .5cm);
	\draw [thick, black, ->] (0, -4cm-\minor-.3cm) -- (3cm-\major-.2cm, -7cm+\minor+.3cm) node [above, sloped, midway] {P2};
	\newcommand{\drawgenVb}[3]{
	\drawgenIV{#1}{#2}{#3}
	\coordinate (v2) at ($(tleft)+(0, .5)$);
	\node[draw,circle,inner sep = 1pt,fill] at (v2) {};
	\draw [thick, black] (tleft) -- (v2);
	}
	\drawgenVb{3,-4}{\major}{\minor}
	\draw [black, dashed] (tleft) -- (v1);
	\draw [thick, black, ->] (3, -4cm-\minor-.3cm) -- (3cm, -7cm+\minor+.7cm) node [right, midway] {P2};
	\newcommand{\drawgenVc}[3]{
	\drawgenIV{#1}{#2}{#3}
	\coordinate (v2) at ($(tright)+(.5, 0)$);
	\node[draw,circle,inner sep = 1pt,fill] at (v2) {};
	\draw [thick, black] (tright) -- (v2);
	}
	\drawgenVc{6,-4}{\major}{\minor}
	\draw [black, dashed] (v2) -- (v1);
	\draw [thick, black, ->] (6, -4cm-\minor-.3cm) -- (3cm+\major+.2cm, -7cm+\minor+.3cm) node [above, sloped, midway] {P2};
	\newcommand{\drawgenVd}[3]{
	\drawgenIV{#1}{#2}{#3}
	\coordinate (v2) at ($(v1)+(.5, 0)$);
	\coordinate (v3) at ($(v1)+(1, 0)$);	
	\node[draw,circle,inner sep = 1pt,fill] at (v2) {};
	\node[draw,circle,inner sep = 1pt,fill] at (v3) {};
	\draw [thick, black] (v2) -- (v3);
	}
	\drawgenVd{11,-4}{\major}{\minor}
	\draw [black, dashed] (tleft) -- (v1);
	\draw [thick, black, ->] (11, -4cm-\minor-.3cm) -- (11cm, -7cm+\minor+.9cm) node [left, midway] {P2};
	\typeD{3,-7}{\major}{\minor} 
	\typeE{11,-7}{\major}{\minor} 
	
	\draw [thick, black, dotted] (tright) arc (0:180: .5cm);
	\draw [thick, black, ->] (3, -7cm-\minor-.3cm) -- (0cm+\major, -11cm+\minor+1.2cm) node [above, sloped, midway] {P1};
	\draw [thick, black, ->] (3, -7cm-\minor-.3cm) -- (3cm, -11cm+\minor+1cm) node [right, midway] {P1};
	\draw [thick, black, ->] (3, -7cm-\minor-.3cm) -- (6cm-\major-.2cm, -11cm+\minor+1cm) node [above, sloped, midway] {P1};
	\draw [thick, black, ->] (11, -7cm-\minor-.3cm) -- (9cm+.2cm, -11cm+\minor+1cm) node [above, sloped, midway] {P1 (forced)};
	\draw [thick, black, ->] (11, -7cm-\minor-.3cm) -- (12cm, -11cm+\minor+1cm) node [above, sloped, midway] {P1 (forced)};
	\newcommand{\drawgenVIIa}[3]{
	\typeD{#1}{#2}{#3}
	\coordinate (v1) at ($(tright)+(0, .5)$);
	\coordinate (v2) at ($(tright)+(.5, .5)$);
	\node[draw,circle,inner sep = 1pt,fill] at (v1) {};
	\node[draw,circle,inner sep = 1pt,fill] at (v2) {};
	\draw [thick, black] (v1) -- (v2);
	}
	\drawgenVIIa{0,-11}{\major}{\minor}
	\draw [black, dashed] (tmiddle) -- (tright);
	\draw [thick, black, ->] (0, -11cm-\minor-.3cm) -- (3cm, -15cm+\minor+.7cm) node [above, sloped, midway] {P2};
	\newcommand{\drawgenVIIb}[3]{
	\typeD{#1}{#2}{#3}
	\coordinate (v1) at ($(tright)+(0, .5)$);
	\node[draw,circle,inner sep = 1pt,fill] at (v1) {};
	\draw [thick, black] (tright) -- (v1);
	}
	\drawgenVIIb{3,-11}{\major}{\minor}
	\draw [black, dashed] (tmiddle) -- (v1);
	\draw [thick, black, ->] (3, -11cm-\minor-.3cm) -- (8cm-.12cm, -15cm+\minor+.7cm)  node [above, sloped, pos=.4] {P2};
	\newcommand{\drawgenVIIc}[3]{
	\typeD{#1}{#2}{#3}
	\coordinate (v1) at ($(tmiddle)+(0, .5)$);
	\node[draw,circle,inner sep = 1pt,fill] at (v1) {};
	\draw [thick, black] (tmiddle) -- (v1);
	}
	\drawgenVIIc{6,-11}{\major}{\minor}
	\draw [black, dashed] (tright) -- (v1);
	\draw [thick, black, ->] (6, -11cm-\minor-.3cm) -- (8cm, -15cm+\minor+.7cm) node [above, sloped, pos=.3] {P2};
	\newcommand{\drawgenVIId}[3]{
	\typeE{#1}{#2}{#3}
	\draw [thick, black] (tright) -- (v1);
	}
	\drawgenVIId{9,-11}{\major}{\minor}
	\draw [black, dashed] (tmiddle) -- (v2);
	\draw [thick, black, ->] (9, -11cm-\minor-.3cm) -- (8cm+.1cm, -15cm+\minor+.7cm) node [right, pos=.6] {P2};
	\newcommand{\drawgenVIIe}[3]{
	\typeE{#1}{#2}{#3}
	\coordinate (v3) at ($(tleft)+(0, .5)$);
	\node[draw,circle,inner sep = 1pt,fill] at (v3) {};
	\draw [thick, black] (tleft) -- (v3);
	}
	\drawgenVIIe{12,-11}{\major}{\minor}
	\draw [black, dashed] (tmiddle) -- (v3);
	\draw [thick, black, ->] (12, -11cm-\minor-.3cm) -- (3cm, -15cm+\minor+.7cm) node [above, sloped, pos=.2] {P2};
	\typeAEdge{3,-15}{\major}{\minor} 
	\typeB{8,-15}{\major}{\minor} 
\end{tikzpicture}
}
\caption{Player 2's strategy for when $C_n$ is type $\mathcal{H}$, $D_n$ is empty, and it is Player 2's turn. The dashed line represents the move that Player 2 will make. Note that not all of Player 1's potential moves are considered, as some moves lead immediately to an eventual cut vertex or 3-regular component once Player 2 makes the dotted line move, thereby forcing Player 1 into the depicted options.}
\label{tree of cases P2}
}
\end{figure}


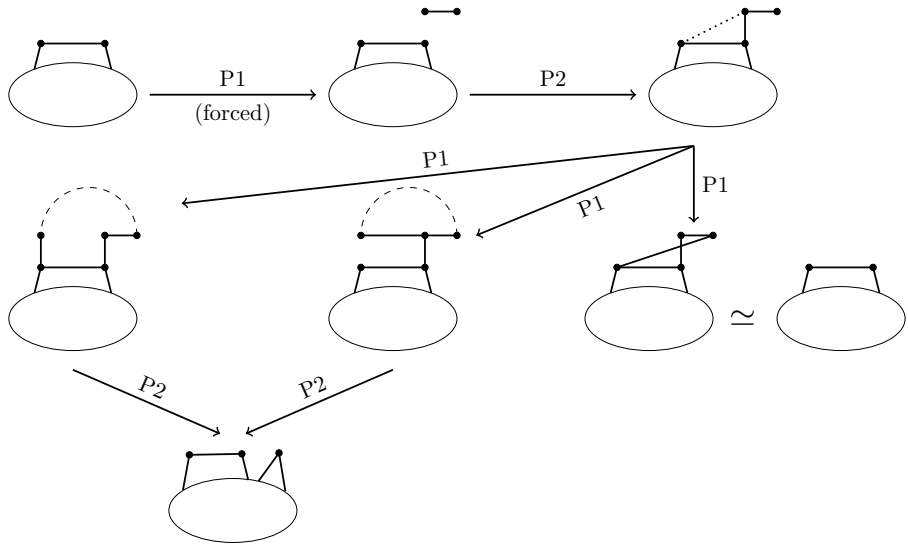
\begin{figure}[h]
\resizebox{1\linewidth}{!}{
\begin{tikzpicture}
	\typeH{0,3}{\major}{\minor}
	\draw [thick, black, ->] (\major+.2cm, 3) -- (\major+2.8cm, 3) node [above, midway] {P1} node [below, midway] {(forced)};
	\newcommand{\drawgenII}[3]{
	\typeH{#1}{#2}{#3}
	\coordinate (v1) at  ($(tright)+(0,.5)$);
	\coordinate (v2) at ($(tright)+(.5,.5)$);
	\node[draw,circle,inner sep = 1pt,fill] at (v1) {};
	\node[draw,circle,inner sep = 1pt,fill] at (v2) {};
	\draw [thick, black] (v1) -- (v2);
	}
	\drawgenII{5,3}{\major}{\minor}
	\draw [thick, black, ->] (5cm+\major+.2cm, 3) -- (5cm+\major+2.8cm, 3) node [above, midway] {P2};
	\newcommand{\drawgenIII}[3]{
	\drawgenII{#1}{#2}{#3}
	\draw [thick, black] (tright) -- (v1);
	}
	\drawgenIII{10,3}{\major}{\minor}
	\draw [thick, black, dotted] (tleft) -- (v1);
	\draw [thick, black, ->] (10cm-.3cm, 3cm-\minor-.3cm) -- (\major+.7cm, 3cm-\minor-1.2cm) node [above, sloped, midway] {P1};
	\draw [thick, black, ->] (10cm-.3cm, 3cm-\minor-.3cm) -- (5cm+\major+.3cm, 3cm-\minor-1.7cm) node [below, sloped, midway] {P1};
	\draw [thick, black, ->] (10cm-.3cm, 3cm-\minor-.3cm) -- (10cm-.3cm, 3cm-\minor-1.5cm) node [right, midway] {P1};
	\newcommand{\drawgenIVa}[3]{
	\drawgenIII{#1}{#2}{#3}
	\coordinate (v3) at ($(tleft)+(0,.5)$);
	\node[draw,circle,inner sep = 1pt,fill] at (v3) {};
	\draw [thick, black] (tleft) -- (v3);
	}
	\drawgenIVa{0,-.5}{\major}{\minor}
	\draw [black, dashed] (v2) arc (0:180: .75cm);
	\draw [thick, black, ->] (0cm, -.5cm-\minor-.3cm) -- (2.3cm, -.5cm-\minor-.3cm-1cm) node [above, sloped, midway] {P2};
	\newcommand{\drawgenIVb}[3]{
	\drawgenIII{#1}{#2}{#3}
	\coordinate (v3) at ($(tleft)+(0,.5)$);
	\node[draw,circle,inner sep = 1pt,fill] at (v3) {};
	\draw [thick, black] (v1) -- (v3);
	}
	\drawgenIVb{5,-.5}{\major}{\minor}
	\draw [black, dashed] (v2) arc (0:180: .75cm);
	\draw [thick, black, ->] (5cm, -.5cm-\minor-.3cm) -- (2.7cm, -.5cm-\minor-.3cm-1cm) node [above, sloped, midway] {P2};
	\newcommand{\drawgenIVc}[3]{
	\drawgenIII{#1}{#2}{#3}
	\draw [thick, black] (tleft) -- (v2);
	\coordinate (iso) at ($(#1)+(\major+2,0)$);
	\node[right, scale = 1.5] at (iso) {$\simeq$} ;
	\typeH{$(#1)+(3,0)$}{#2}{#3}
	}
	\drawgenIVc{9,-.5}{\major}{\minor}
	\typeB{2.5,-3.5}{\major}{\minor}

\end{tikzpicture}
}
\caption{Player 2's strategy for when $C_n$ is type $\mathcal{H}$, $D_n$ is empty, and it is Player 1's turn. The dashed line represents the move that Player 2 will make, and the dotted line represents a forcing move from Player 2.}
\label{tree of cases P1}
\end{figure}

\section{Case Studies} \label{case studies}

In order to prove the previously used lemmas, we define the following types of graphs.
\begin{definition}
Let $G$ be a connected graph. We say that $G$ is \emph{type $\mathcal{X}$} if $G$ is 3-regular except for two vertices of degree 1 which are not adjacent.
\end{definition}

\begin{definition}
Let $G$ be a connected graph. We say that $G$ is \emph{type $\mathcal{Y}$} if $G$ is 3-regular except for one vertex of degree 1 and one vertex of degree 2 which are not adjacent.
\end{definition}

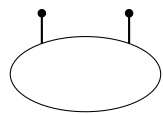
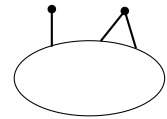
\begin{figure}[h]
\centering
\begin{subfigure}[h]{0.3\textwidth}
\centering
\begin{tikzpicture}
	\typeX{0,0}{\major}{\minor}
\end{tikzpicture}
\subcaption{Type $\mathcal{X}$}
\end{subfigure}
\hspace{3cm}
\begin{subfigure}[h]{0.3\textwidth}
\centering
\begin{tikzpicture}
	\typeY{2,0}{\major}{\minor}
\end{tikzpicture}
\subcaption{Type $\mathcal{Y}$}
\end{subfigure}
\caption{Depictions of Type $\mathcal{X}$ and Type $\mathcal{Y}$ graphs.}
\label{types X and Y ex}
\end{figure} 

\begin{proof}[Proof of Lemma \ref{E decreases}]

We use the following lemma in this proof, to be proven later in this section:
\begin{lemma}
Assume that $C_0$ is not a type $\mathcal{X}$ graph.
Suppose that there is some $n > 1$ such that $C_n$ is type $\mathcal{H}$ and $C_n$ does not have an eventual cut vertex. Then, there exists $m < n$ such that Player 1 drew an edge within $C_m$. 
\label{avoid type H}
\end{lemma}

We can assume that $m > n+1$ is minimal such that $C_m$ is type $\mathcal{H}$ with no eventual cut vertex.
In the first two moves after $C_n$, Player 2's strategy is to draw an edge from either $C_n$ to $D_n$ or $C_{n+1}$ to $D_{n+1}$, depending on the parity of $n$, and as long as $D_n$ or $D_{n+1}$ is not empty. Otherwise, Player 2 will follow either Figure \ref{tree of cases P2} or Figure \ref{tree of cases P1}, from which an eventual cut vertex or 3-regular component will be made.
Assuming $D_n$ and $D_{n+1}$ are not empty, then Player 1 can increase $E(D_n)$ or $E(D_{n+1})$ by at most 2, which is accomplished by drawing an edge between two isolated vertices.
Noting that $F(C_n) = 2$, we have
\[ F(C_{n+2}) + E(D_{n+2}) \leq 2 + y+2 = y+4 \]
Note that if $m=n+2$, then Player 1's move was not to draw an edge between two isolated vertices, but to draw within $C_{n+1}$, in which case $D_m$ is the same as $D_n$, except with one less component. Thus, $E(D_m) < E(D_n)$. Now, if $m \neq n+2$, then since $C_m$ is type $\mathcal{H}$ with no eventual cut vertex, lemma \ref{avoid type H} guarantees the existence of $k$ with $n+2 \leq k \leq m$ such that Player 1 played within $C_{k-1}$. 
Let this $k$ be minimal. By lemma \ref{F+E non-increasing}, we know that 
\[ F(C_{k-1}) + E(D_{k-1}) \leq 
\begin{cases}
y+4  & \text{ if } n \text{ even} \\
y+2  & \text{ if } n \text{ odd } 
\end{cases}
\]
We case by the parity of $n$ as lemma \ref{F+E non-increasing} gives us that $F(C_{k-2}) + E(D_{k-2}) \leq y+4 $ for when $n+2$ is odd.
Now, once Player 1 plays within $C_{k-1}$, then we get $F(C_k) = F(C_{k-1})-2$. If $k = m$, then we have
\[ F(C_m) + E(D_m) = F(C_{k-1}) - 2 + E(D_{k-1}) \leq 
\begin{cases}
y+2 & \text{ if } n \text{ even } \\
y & \text{ if } n \text{ odd }
\end{cases}
\]
Since $F(C_m) = 2$, then this implies that $E(D_m) \leq y$ if $n$ is even, and $E(D_m) < y$ if $n$ is odd. If $k \neq m$, then Player 2's strategy is to play within $C_k$, which results in $F(C_{k+1}) = F(C_k) - 2$. 
Then, we have
\[ F(C_{k+1}) + E(D_{k+1}) = F(C_{k-1}) -4 + E(D_{k-1}) \leq y \]
Again applying lemma \ref{F+E non-increasing}, we get that $ F(C_m) + E(D_m) \leq y $ which in turn implies that $E(D_m) < y$.

Thus, we see that $E(D_m) < y$ in all cases except when both $n$ is even and $k=m$, in which case $E(D_m) \leq y$.
Suppose this is satisfied, and suppose further that there is $m' > m +1$ with $C_{m'}$ type $\mathcal{H}$ without an eventual cut vertex. Then, since $k=m$ and $k$ is odd, it must be that $E(D_{m'}) < E(D_m)$.

\end{proof}

\begin{proof}[Proof of Lemma \ref{F+E non-increasing}]
Note that Player 2's move decreased $F$ by 2 and kept $E$ constant, i.e.
\[ F(C_{2n}) + E(D_{2n}) = F(C_{2n-1}) + E(D_{2n-1}) - 2 \]
Player 1 now has several options of where to draw an edge. We keep track of $F+E$ in each option. \newline
\begin{tabular}{p{7.8cm} | l}
Player 1's move & $F(C_{2n+1}) + E(D_{2n+1})$ \\ \hline
(a) draw an edge between two isolated vertices. & $F(C_{2n}) + E(D_{2n}) + 2$ \\
(b) draw an edge from a component in $D_{2n}$ to an isolated vertex. & $F(C_{2n}) + E(D_{2n}) + 1$ \\
(c) draw an edge from $C_{2n}$ to an isolated vertex. & $F(C_{2n}) + E(D_{2n}) + 1$ \\
(d) draw an edge from a component in $D_{2n}$ to $C_{2n}$. & $F(C_{2n}) + E(D_{2n})$ \\
(e) draw an edge within $D_{2n}$. & $F(C_{2n}) + E(D_{2n}) - 2$ \\
(f) draw an edge within $C_{2n}$. & $F(C_{2n}) + E(D_{2n}) - 2$ \\[0.5em]
\end{tabular} \\
Thus, in all the options we see that 
\[ F(C_{2n+1}) + E(D_{2n+1}) \leq F(C_{2n-1}) + E(D_{2n-1}) \]
and equality is attained if and only if Player 1 performs option (a), which is to draw an edge between two isolated vertices.
\end{proof}

\begin{proof}[Proof of Lemma \ref{type A}]
Let the vertices of degree 2 in one of the components that is not a single edge be $p,q$ and let $a,b$ be the vertices of degree 1 in the single edge. If it is Player 2's turn, then Player 2's strategy is to draw an edge between $p$ and $q$, thus creating a 3-regular component. If it is Player 1's turn, first consider the case where Player 1 plays on neither $p$ nor $q$. Then Player 2's strategy is to draw edge between $p$ and $q$, which again makes a 3-regular component.

Otherwise, say Player 1 draws an edge from $p$ to a vertex $v$, where $v$ could be $a, b$, an isolated vertex, or to $p'$, where $p'$ is a vertex of degree 2 in another connected component. In the first three cases, the degree of $v$ is at most 2, so Player 2's strategy is to draw an edge from $q$ to $v$, which makes $v$ an eventual cut vertex. In the last case, Player 2's strategy is to draw an edge from $q$ to $q'$, where $q'$ is the other vertex of degree 2 in the same component as $p'$. This creates a 3-regular component.
\end{proof}

\begin{proof}[Proof of Lemma \ref{type B}]
Call the two adjacent vertices of degree 2 $p,q$ and the remaining degree 2 vertex $x$. Then observe that an edge between $p$ and $x$ makes $q$ an eventual cut vertex, and similarly an edge between $q$ and $x$ makes $p$ an eventual cut vertex. Thus, if either move is available on Player 2's turn, then Player 2 can create an eventual cut vertex. In order to avoid this, the only possible move for Player 1 is to draw an edge off of $x$ to some isolated vertex $y$. Player 2's strategy is now to draw an edge between $y$ and $p$, which results in a component with two vertices of degree 2 which are not adjacent. 

From here, Player 1 draws an edge. If the edge is on two free vertices, then Player 2 can draw an edge from $q$ to $y$, making a 3-regular component. Otherwise, say that Player 1 draws an edge from $q$ to an isolated vertex $z$. At this point, Player 2 draws an edge from $y$ to $z$, making $z$ an eventual cut vertex. 
\end{proof}

\begin{proof}[Proof of Lemma \ref{avoid type H}]
Suppose that for all $n$, Player 1 never plays within $C_n$.

Assume for contradiction that $C_{2n-1}$ is type $\mathcal{H}$ (so that Player 1 drew the edge to make it type $\mathcal{H}$) and does not have an eventual cut vertex. Consider the graph $C_{2n-2}$. First note that Player 1 cannot have drawn an edge from $C_{2n-2}$ to an isolated vertex, since then $C_{2n-1}$ would have a degree 1 vertex and hence not by type $\mathcal{H}$. So, Player 1 drew an edge from $C_{2n-2}$ to some component $D_{2n-2}^j$. Now note that $C_{2n-2}$ can have at most one degree 1 vertex, since otherwise $C_{2n-1}$ would have a degree 1 vertex. Moreover, if $C_{2n-2}$ has exactly one degree 1 vertex, then it must also have a degree 2 vertex, otherwise the degree 1 vertex would be an eventual cut vertex.

Suppose $C_{2n-2}$ has a degree 1 vertex $v$ and a degree 2 vertex $w$. If $v \sim w$, then $w$ is an eventual cut vertex, and so $C_{2n-1}$ has an eventual cut vertex. Otherwise, since $v \not \sim w$, then $C_{2n-1}$ will either have an extra degree 1 vertex or an extra degree 2 vertex, and hence not be type $\mathcal{H}$. Thus, $C_{2n-2}$ has no degree 1 vertices.

Let $v, w$ be the vertices of degree 2 in $C_{2n-1}$. Since $C_{2n-2}$ has no degree 1 vertices, then $v, w$ must also be degree 2 in $G_{2n-2}$. 
Since they are adjacent, then $v, w$ must be from the same component, either $C_{2n-2}$ or $D_{2n-2}^j$. Say $v, w \in C_{2n-2}$. Then, $D_{2n-2}^j$ only has one degree of freedom, and so $D_{2n-2}^j$ and hence $C_{2n-1}$ has an eventual cut vertex.

Thus, in all cases, $C_{2n-1}$ is either not type $\mathcal{H}$ or has an eventual cut vertex, and hence we get a contradiction. Thus, if $C_n$ is type $\mathcal{H}$ without an eventual cut vertex, then $n$ must be even. However, if Player 1 never plays within $C_n$, then we claim that $C_{2n}$ is never type $\mathcal{H}$ without an eventual cut vertex.
We prove this using the following three subclaims:
\begin{enumerate}
\setlength{\itemindent}{1cm}
\setlength{\itemsep}{-.2em}
\item[\underline{Subclaim 1}]
If $C_{2n}$ is type $\mathcal{H}$ (so that Player 2 drew the edge to make it type $\mathcal{H}$) and does not have an eventual cut vertex, then $C_{2n-1}$ must be type $\mathcal{X}$.
\item[\underline{Subclaim 2}]
If $C_{2n-1}$ is type $\mathcal{X}$, then $C_{2n-2}$ is type $\mathcal{Y}$.
\item[\underline{Subclaim 3}]
If Player 2 follows the strategy from above, then $C_{2n-2}$ will never be type $\mathcal{Y}$ for all $n \ne 0$. 
\end{enumerate}

\noindent \underline{Proof of Subclaim 1:}

Since $C_{2n-1}$ was not type $\mathcal{H}$, Player 2's move always decreases $F(C_{2n-1})$ by two. Then, the only components that would force Player 2 to make a type $\mathcal{H}$ graph must have four degrees of freedom. The possible options for $C_{2n-1}$ are:
\begin{enumerate}
\setlength{\itemsep}{-.2em}
\item 3-regular except for four vertices of degree 2.
\item 3-regular except for one vertex of degree 1 and two vertices of degree 2. 
\item 3-regular except for two vertices of degree 1 (i.e. $C_{2n-1}$ is type $\mathcal{X}$).
\end{enumerate}
In the first case, Player 2's strategy leaves two vertices of degree 2 that are not adjacent, and hence $C_{2n}$ would not be type H.

In the second case, let $u$ be the vertex of degree 1 and let $v, w$ be the vertices of degree 2. If $u$ is not adjacent to either $v$ or $w$, then Player 2's strategy is to draw the edge $u \sim v$, which leaves two vertices of degree 2 that are not adjacent, namely $u, w$, and hence $C_n$ is not type $\mathcal{H}$. If $u$ is adjacent to one of $v, w$, say $u \sim v$, then Player 2's strategy is to draw the edge $v \sim w$, if not already adjacent, and to draw the edge $u \sim w$ if $v, w$ were already adjacent. In the former case, $C_{2n}$ is not type $\mathcal{H}$, and in the latter case, $C_{2n}$ has an eventual cut vertex.

The third case forces $C_{2n}$ to be type $\mathcal{H}$. 

\noindent \underline{Proof of Subclaim 2:}

Note that under Player 2's strategy, $C_{2n-2}$ will not be type $\mathcal{X}$, since Player 2 will always draw an edge between two degree 1 vertices if at least two exist. Thus, $C_{2n-2}$ cannot be type $\mathcal{X}$, and Player 1's move must have made $C_{2n-1}$ type $\mathcal{X}$. 

Now note that $C_{2n-1}$ does not have a degree 2 vertex. Thus, $C_{2n-2}$ can have at most one degree 2 vertex, since otherwise a degree 2 vertex would be in $C_{2n-1}$. Similarly, we also have that $C_{2n-1}$ must have at least one degree 2 vertex, since otherwise drawing an edge from a degree 1 vertex in $C_{2n-2}$ to any other vertex would make the degree 1 vertex now a degree 2 vertex. Also, we can assume that $C_{2n-2}$ has at least one degree 1 vertex, otherwise $C_{2n-2}$ and hence $C_{2n}$ would have an eventual cut vertex. Thus, $C_{2n-2}$ must have at least one degree 1 vertex and exactly one degree 2 vertex, and Player 1's move must have been to draw an edge from this degree 2 vertex to some other vertex.

Let $v$ be the degree 2 vertex. If Player 1 draws an edge from $v$ to an isolated vertex, then we have $F(C_{2n-1}) = F(C_{2n-2}) + 1$, and if Player 1 draws an edge from $v$ to some component $D_{2n-2}^j \subseteq D_{2n-2}$, then we have $F(C_{2n-1}) = F(C_{2n-2}) + E(D_{2n-2}^j)$. Moreover, we can assume that $E(D_{2n-2}^j) \geq 0$, since otherwise $D_{2n-2}$ and hence $C_n$ has an eventual cut vertex. 

Thus, in either case, we have that $F(C_{2n-2}) \leq F(C_{2n-1}) = 4$. This implies that $C_{2n-1}$ cannot have more than one degree 1 vertex, and so must have exactly one degree 1 vertex and exactly one degree 2 vertex, i.e. is type $\mathcal{Y}$.

\noindent \underline{Proof of Subclaim 3:}

If $C_{2n-2}$ is type $\mathcal{Y}$, then $F(C_{2n-2}) = 3$. Since Player 2's strategy is always to draw edges within $C_{2n-3}$, it suffices to consider all components $C_{2n-3}$ such that $F(C_{2n-1})=5$. Then consider those graphs that have at least one vertex of degree 1 (otherwise $C_{2n-2}$ will not by type $\mathcal{Y}$). Either $C_{2n-1}$ has one vertex of degree 1 and three of degree 2, or it has two vertices of degree 1 and one vertex of degree 2. In the first case, according to the strategy, Player 2 connects the vertex of degree 1 to a vertex of degree 2. In the second case, the strategy says to connect the two vertices of degree 1 together. In either case, $C_{2n-2}$ is not type $\mathcal{Y}$, since $C_{2n-2}$ has no degree 1 vertices. 
\end{proof}

\begin{lemma}
Let $C$ be a component that is 3-regular except for four vertices of degree 2. Then, there is a labeling of the four vertices, $u, v, p, q$, such that $u \not \sim v$ and $p \not \sim q$.
\label{square}
\end{lemma}

\begin{proof}[Proof of Lemma~\ref{square}]
Vertex $u$ can either be adjacent to none, one, or two of $v,p,q$.
Consider the following proof by cases:

\noindent \underline{Case 1:} $u$ is adjacent to none of $v, p, q$. \\
Then we also know that not all of $v,p,q$ are adjacent to each other, as this would create a triangle and thus be disconnected from $C$. Then, without loss, $p \not \sim q$ and by assumption $u \not \sim v$. 

\noindent \underline{Case 2:} $u$ is adjacent to one of $v, p, q$. \\
Without loss, say $u \sim v$. Then as $v$ is degree 2 and already adjacent to $u$, it can only be adjacent to at most one of $p,q$, say $v \sim p$. Thus $v \not \sim q$ and $u \not \sim p$. 

\noindent \underline{Case 3:} $u$ is adjacent to two of $v, p, q$. \\
Without loss, say $u \sim v$ and $u \sim q$. So, $u \not \sim p$. However, we know that $v \not \sim q$, because this would make a triangle with vertices $u,v,q$. 

Hence, in any case, we can always find a pairing such that  $u \not \sim v$ and $p \not \sim q$. \end{proof}

\section*{Acknowledgments}

We thank Wesley Pegden for suggesting the question and for helpful discussions.

\end{document}